\documentclass[12pt, a4paper]{amsart}
\usepackage{amsfonts}
\usepackage{amsthm}
\usepackage{amsmath}
\usepackage{amscd}
\usepackage[latin2]{inputenc}
\usepackage{t1enc}
\usepackage[mathscr]{eucal}
\usepackage{indentfirst}
\usepackage{graphicx}
\usepackage{graphics}
\usepackage{pict2e}
\usepackage{epic}
\numberwithin{equation}{section}
\usepackage[margin=2.9cm]{geometry}
\usepackage{hyperref}
\usepackage{lipsum}

\usepackage[dvipsnames,usenames]{color}

\theoremstyle{plain}
\newtheorem{Th}{Theorem}[section]
\newtheorem{Th*}{Theorem}
\newtheorem*{Theorem*}{Theorem}

\newtheorem{Cor*}{Corollary}
\newtheorem*{Corollary*}{Corollary}

 \theoremstyle{definition}

\newtheorem{?}[Th]{Problem}

\begin{document}

\title{ Hausdorff operators on Fock Spaces}

\author{P. Galanopoulos
  \and
  G. Stylogiannis
}

\newcommand{\Addresses}{{
  \bigskip
  \footnotesize

  P.~Galanopoulos, \textsc{Department of Mathematics, Aristotle University of Thessaloniki,
    Thessaloniki, Greece 54 124}\par\nopagebreak
  \textit{E-mail address}, P.~Galanopoulos: \texttt{petrosgala@math.auth.gr}

  \medskip

  G.~Stylogiannis, \textsc{Department of Mathematics, Aristotle University of Thessaloniki,
    Thessaloniki, Greece 54 124}\par\nopagebreak
  \textit{E-mail address}, G.~Stylogiannis: \texttt{stylog@math.auth.gr}

  }}

\begin{abstract}
 Let $\mu$ be a  positive Borel measure on the positive real axis. We study  the integral operator  \\
 $$
 \mathcal{H}_{\mu}(f)(z)=\int_{0}^{\infty}\frac{1}{t}f\left(\frac{z}{t}\right)\,d\mu(t),\quad z\in \mathbb{C}\,,
 $$
 \\
 acting on the Fock spaces $F^{p}_{\alpha}$, $p\in [1,\infty],\,\alpha >0$. Its action  is easily seen to be a coefficient
multiplication  by the moment sequence
  $$
  \mu_n= \int_{1}^{\infty}\frac{1}{t^{n+1}}\,d\mu(t) .
  $$
   We  prove that
   \begin{equation*}
 ||\mathcal{H}_{\mu}||_{F^{p}_{\alpha}\to F^{p}_{\alpha}}=\sup_{n\in\mathbb{N}}\mu_n,\,\,\,\,\,1\leq p\leq \infty\,\,.
\end{equation*}
\\
A little\,-\,o \,condition  describes the compactness of $\mathcal{H}_{\mu}$ on every $F^{p}_{\alpha},\,p\in (1,\infty )$. In
addition, we completely characterize the Schatten class membership of $\mathcal{H}_{\mu}$.
\end{abstract}

\maketitle

\section{introduction}
Let  $\mu$ be a  positive Borel measure on $ (0,\infty)$ and $ H(\mathbb C)$ the space
of holomorphic functions in the complex plane. We formally consider the integral

 \begin{equation}\label{Df Hausd 2}
\mathcal{H}_{\mu}(f)(z)=\int_{0}^{\infty}\frac{1}{t}f\left(\frac{z}{t}\right)\,d\mu(t),\quad z\in \mathbb{C},
\end{equation}
where $f\in  H(\mathbb C) $. This is  the  Hausdorff operator, induced by the measure $\mu$.

Hausdorff operators have a long history. The classical discrete setting is as follows, see \cite{Har}.
 Let $ \{\mu_n\}_n $ be a sequence, and consider the lower triangular  matrix with entries
$$
c_{m,k}=\begin{cases} \binom{m}{k} \Delta^{m-k}\,\mu_{k},  \,\, k\leq m\,,\\
 0, \quad \quad \quad \quad \quad k> m\,,
\end{cases}
$$
where
\begin{equation*}
\Delta_0\,\mu_n= \mu_n-\mu_{n+1} \quad and \quad
\Delta^k \mu_n =\Delta(\Delta^{k-1}\mu_n)\,.
\end{equation*}
\\
This is the Hausdorff matrix and acts on sequences  $\{ s_n\}_n$ as follows
\\
\begin{equation}\label{HausSum}
t_m = \sum_{k=0}^m\, c_{m,k}\, s_k \,.
\end{equation}
Under the regularity condition
\\
\begin{equation}\label{regularity}
\Delta^n t_0 = \mu_n\, \Delta^n s_0\,,
\end{equation}
\\
 the transformation (\ref{HausSum}) is the generalization of  classical methods of summation such as the Ces\'aro method corresponding to
 $\mu_n=\frac {1}{n+1}$ and the  Holder method.   Cesaro summation on spaces of analytic functions was studied in
\cite{Si1}, \cite{Si2} and the Hausdorff method in \cite{GaSi},\cite{GaPa}. For Hausdorff matrices acting on sequence spaces
see for example  \cite{De},\cite{Le}, \cite{Rh} and the references therein.

An continuous version of the above is as follows, see  \cite{Har}. Let $f$ be a real variable function defined on $[0,\infty)$ such
that

$$
f(x) = \sum_{n\geq 0}\, \frac{f^{(n)}(0)}{n!}\,x^n = \sum_{n\geq 0}\, a_n \,x^n\,,\qquad \text{for x close to}\,\, 0\,
$$
and $\{\mu_n \}_n$ be a sequence. Consider a function $g$,  the image of $f$, such that
\\
$$ g^{(n)}(0) =\,\mu_n \, f^{(n)}(0). $$
\\
This is the natural continuous analogue of (\ref{regularity}). This operator $f \to g$ is represented by the series
\\
\begin{equation}\label{CoefMult}
g(x) = \sum_{n\geq 0}\, \mu_n\,a_n \,x^n \,.
\end{equation}
A particular interesting case  of $\mu_n$  is a moment sequence
$$
\mu_n = \int t^n \, \phi(t)\,dt\,,
$$
\\
which formally  induces the integral operator
$$
g(x)=\int\, f(tx)\, \phi(t)\,dt\,.
$$
A typical example is the case
$$\phi(t)=1-(1-t)^k,\,\,\,\quad t\in(0,1),$$
for $k>0$. This  corresponds to the $(C,k)$ integral means of  $f$.

In \cite{Ge}, this idea is applied to the space of Fourier-Stieltjes transform of positive Borel measures $\nu $ on $\mathbb R $,
to obtain the Hausdorff   operator
\begin{equation}\label{FST}
T(\nu)(x)= \int_{\mathbb R }\, \widehat{\nu(tx)} \,\phi(t)\,dt\,, \quad x\in \mathbb R,
\end{equation}
 where $\phi \in L^1(\mathbb R)$.  See also \cite{Gol}.

This idea was pursued further to spaces of functions in \cite{LiMo1}. We  present briefly their main result, which has been the
motive, to extend the Hausdorff operator theory to spaces of analytic functions. If in (\ref{FST}) ,
 instead of a measure, we use  an $f$ of the Hardy space of the real line $H^1(\mathbb R)$  then the function

\begin{equation}\label{HausOp}
\mathcal{H}_{\phi}(f)(x)=\int_{\mathbb R}\, \frac{1}{|t|}\,\,f\left(\frac{x}{t}\right)\,\phi(t)\,dt\,, \quad x\in \mathbb R
\end{equation}
\\
has the property
\\
$$
\widehat{\mathcal{H}_{\phi}(f)}(x)=\int_{\mathbb R}\, \widehat{f(tx)} \,\phi(t)\,dt\,, \quad x\in \mathbb R \,.
$$
\\
This is the Hausdorff transformation of $f$. An interesting viewpoint of the Hausdorff operator theory in $L^1(\mathbb R)$ is,
that  for specific choices of $\phi$, we can get the Hardy  and the Ces\' aro integral operators. On the other hand is strongly
connected to the Hilbert transform. All these are made clear through a series of papers. See \cite {LiMo1}, \cite{ChFaLi},
\cite{ChFaZh}, \cite{FaLi}, \cite{GiMo}, \cite{Ka}. For more information on the subject see  the review articles \cite{Li} and
\cite{ChFaWa} and the references there in.

Recently, in  \cite {HuKyQu2}  the authors  considered the action of
 \begin{equation}\label{Df Hausd 1}
\mathcal{H}_{\phi}(f)(z)=\int_{0}^{\infty}\,f\left(\frac{z}{t}\right)\,\frac{\phi(t)}{t}\,dt\,\,,
\end{equation}
on the Hardy spaces of the upper half plane,
where $\phi=\phi(t)$ is a locally integrable function in $(0,\infty)$.
This is the complex variable version of (\ref{HausOp}).
The question under discussion is
what conditions  should we suppose on $\phi$ in order the operator (\ref{Df Hausd 1}) to be bounded.
This question, for the Bergman spaces of the upper half plane, is answered by \cite{St}.

In this article we  study   the Hausdorff operator $\mathcal{H}_{\mu}$, as  defined   in (\ref{Df Hausd 2}), on  the Fock
spaces. Let  $1\leq p <\infty$ and $\alpha >0$. An entire function $f$ belongs  to  the Fock space $F^{p}_{\alpha}$ if
$$
||f||_{p,\alpha}^{p} = \frac{\alpha p}{2\pi}\int_{\mathbb{C}}\left|f (z)e^{-\frac{\alpha}{2} |z|^{2}}\right|^{p} dA(z) \,< \, \infty\, ,
$$
where $dA(z)$ stands for the Lebesgue area measure.
If $p=\infty$ then the space $F^{\infty}_{\alpha}$ consists of those entire functions $f$ such that
$$
||f||_{\infty,\alpha}\,=\, \sup_{z\in\mathbb{C}}\,   {|f (z)|e^{-\frac{\alpha}{2} |z|^2}  }\, < \, \infty\,.
$$
In section $2$  we deal with the boundedness. We make clear
that  the operator is well defined on Fock spaces if
\begin{equation}\label{well definition1}
\mu(0,1)=0
\end{equation}\\
and
\begin{equation}\label{well definition2}
\sup_{n \in {\mathbb{N}}}\, \int_1^{\infty}\, \frac {1}{t^{n+1}}\,d\mu(t)\,< \infty\,.
\end{equation}\\
Moreover, under these conditions, the operator can be written as
$$
\mathcal{H}_{\mu}(f)(z)= \sum_{n\geq 0}\, \int_1^{\infty}\, \frac {1}{t^{n+1}}\,d\mu(t)\,\, a_n\,\,z^n
$$
where $ f(z)=\sum_{n\geq 0} \,a_n\, z^n \,\in F^{\infty}_{\alpha}\,.$ Comparing the latter representation
with (\ref{CoefMult})
we realize that the Hausdorff operator $\mathcal{H}_{\mu}$ on the Fock spaces is
actually the operator, originally introduced  by Hardy for
 $$
 \mu_n=\int_1^{\infty}\, \frac {1}{t^{n+1}}\,d\mu(t)
 $$
 and for the complex plane instead of the real line. Notice that  (\ref{CoefMult}) expresses a  coefficient
  multiplier relation. So we actually deal with  a coefficient multiplier
 problem on Fock spaces. The   coefficient multiplier problem  on spaces of analytic functions
 is an important researcher area several questions remain open. A standard reference on the subject  is \cite{JVA}.\\

We prove the following:
\begin{Th*}
	Let $1\leq p\leq  \infty$, $\alpha>0$ and $\mu$ be a positive Borel measure on $(0,\infty)$ such that:
	(\ref{well definition1}),(\ref{well definition2}) hold. Then the  Hausdorff operator $\mathcal{H}_{\mu}$
 is bounded on $F^{p}_{\alpha}$\,.
	Moreover \\
	\begin{equation*}
	||\mathcal{H}_{\mu}||_{F^{p}_{\alpha}\to F^{p}_{\alpha}}=\sup\left\{\int_{1}^{\infty}\frac{1}{t^{n+1}}\,d\mu(t)\,: n\in\mathbb{N}\right\}
	\,,\,\,\,\,\,1\leq p\leq \infty\,\,.
	\end{equation*}
\end{Th*}
As an application to the coefficient multiplier problem, Theorem 1 implies that
\begin{Cor*}
If $\mu$ is a positive Borel measure on  $(0,\infty)$ that  fulfils the properties
	(\ref{well definition1}),(\ref{well definition2}), then the sequence
	\begin{equation*}
	\mu_n = \int_{1}^{\infty}\frac{1}{t^{n+1}}\,d\mu(t)\, , \quad n\in \mathbb N,
	\end{equation*}
	is a coefficient multiplier for the Fock spaces $F^{p}_{\alpha}, \,p\in[1,\infty]$. In other words,
	if $ f(z)=\sum_{n\geq 0} \,a_n\, z^n \,\in F^{p}_{\alpha}\,$
	 then
	$$
	M_{\mu_n}(f)(z)= \sum_{n\geq 0}\, \mu_n\, a_n\,z^n \,\in F^{p}_{\alpha}.
	$$
	\end{Cor*}

In section $3$ we consider the compactness of  $\mathcal H_{\mu}$.
Let $\mathbb X$ be a Banach  space and
$T$ a bounded operator on $\mathbb X$. We say that $T$ is compact, if and only  if the image of
the unit ball  is compact in the norm topology of the space. Firstly, we handle  the case of  $F^{2}_{\alpha}$ space.
Then, using   the  one-side compactness Theorem 9 of  \cite{CwKal}, we pass from  $F^{2}_{\alpha}$ to $F^{p}_{\alpha}$,\, $1<p<\infty$.
The outcome is:
\begin{Th*}
	Let $p\in(1,\infty)$ and $\alpha>0$. Assume that $\mu$  is a positive Borel measure on $(0,\infty)$ such that
	(\ref{well definition1}),(\ref{well definition2}) are true. Then the  Hausdorff operator $\mathcal{H}_{\mu}$ is compact   on $F^{p}_{\alpha}$  if and only if
	\\
	\begin{equation*}
	\lim_{n\to \infty}\int_{1}^{\infty}\frac{1}{t^{n+1}}\,d\mu(t)=0.
	\end{equation*}
\end{Th*}
\bigskip
In the last  section, we   study  the Schatten classes of $\mathcal H_{\mu}$. Recall  that any compact operator on a Hilbert
space can be decomposed as
$$
T(x)=\sum_{n}\, {\lambda}_{n}\, {\langle x, e_n\rangle}_{H}\, {\sigma}_n\,,\quad x \in H
$$\\
where $\{ e_n\}$,  $\{ \sigma_{n} \}$ are orthonormal sets in $H$ and $\{\lambda_n\}$ is the sequence of  the singular
values of $T$.
 By ${\mathcal S}_p (H) $,  $p\in (0,\infty)$, we denote the Schatten $p-$class of operators acting on $H$, which
  consists of those compact operators $T$ on $H$  whose  sequence of singular numbers $\lambda_n$ belongs
to  $l^p$.
Moreover, it is true that if $\lambda_n$ are the singular numbers of an operator T, then
$$
\lambda_n=\lambda_n (T)=\inf \left\{\|T-K\| : \mbox{rank } K \leq n \right\}\,.
$$
Thus finite rank operators  belong to every ${\mathcal S}_p (H) $.   In some sense,  the membership of an
operator in ${\mathcal S}_p (H) $ measures
the size of the operator. For $1\leq p <\infty,$\, ${\mathcal S}_p (H) $ is a Banach space with the norm
$$
\|T\|^p=\sum_{n}|\lambda_n|^p\,,
$$
while for $0<p<1$ is a complete metric space. For more information on the subject, we refer  to  \cite{Zh1}.
Concluding,  the membership of  $\mathcal{H}_{\mu}$ in   $\mathcal{S}_{p}={\mathcal S}_p (F^{2}_{\alpha}) $ is described by the following:
\begin{Th*}
	Let $\mu$ be a positive Borel measure on $(0,\infty)$ such that
	(\ref{well definition1}),(\ref{well definition2}) are satisfied. Then the  Hausdorff operator $\mathcal{H}_{\mu}$ is is in Schatten class $\mathcal{S}_{p}$, $0<p<\infty$, if and only if
	\begin{equation*}
	\sum_{n=0}^{\infty}\left(\int_{1}^{\infty}\frac{1}{t^{n+1}}\,d\mu(t)\right)^{p}<\infty.
	\end{equation*}
\end{Th*}

\section{Boundedness}
In the beginning we recall the definition of the spaces and we list a few properties we will use.
Let  $ p \in (0,\infty)$ and $\alpha >0$. An $f \in H(\mathbb C)$ belongs  to  the Fock space $F^{p}_{\alpha}$
if
\begin{equation}\label{norm}
||f||_{p,\alpha}^{p} = \frac{\alpha p}{2\pi}\int_{\mathbb{C}}\left|f (z)e^{-\frac{\alpha}{2} |z|^{2}}\right|^{p} dA(z) \,< \, \infty\, ,
\end{equation}
where $dA(z)$ stands for the Lebesgue area measure.

It is true that, for each $z \in \mathbb{C}$,
\begin{equation}\label{pointwise growth}
|f (z)| \leq ||f||_{p,\alpha}\,\, e^{\frac{\alpha}{2} |z|^2} \, , \quad  \forall f\in F^p_{\alpha}\,.
\end{equation}
In other words, the point evaluation functionals are bounded on each $F^{p}_{\alpha}$.

Based on these growth estimates, it is natural to define
the space $F^{\infty}_{\alpha}$  \, that is  the space of  all the entire functions $f$ such that
$$
||f||_{\infty,\alpha} = \sup_{z\in\mathbb{C}}\,   {|f (z)|e^{-\frac{\alpha}{2} |z|^2}  }\, < \, \infty\,.
$$

It is not difficult to verify  that  the Fock spaces employed with the norm (\ref{norm}),  form a chain of  Banach spaces  with the following containment property of strict inclusions
 \begin{equation}\label{inclusions}
 F^{p}_{\alpha}  \subset  F^{q}_{\alpha} \subset F^{\infty}_{\alpha} \,,\quad 1 \leq p < q < \infty\,.
 \end{equation}
Moreover, $F^2_{\alpha}$ is a Hilbert space.\,  In this case the norm of an
   $f(z)=\sum_{n\geq 0}\, a_n z^n\, \in F^2_{\alpha}$ can be
 equivalently expressed  in terms of it's Taylor coefficients as
\begin{equation}\label{norm Hilbert}
\|f\|^2_{F^2_{\alpha}}\,=\, \sum_{n\geq 0}\, |a_n|^2\,\frac{n!}{\alpha^n}\,.
\end{equation}
 A standard reference for the theory of $F^{p}_{\alpha}$ spaces is \cite{Zh2}.

Having in mind the above and before anything else we verify the well definition of the integral transform (\ref{Df Hausd 2})  on the Fock spaces\,. Consider a $z\in \mathbb C$ and an $f\in F^{\infty}_{\alpha}$, which is the largest space of all  according to  (\ref {inclusions}).  Taking into account  (\ref{pointwise growth}), we get that
\begin{align*}
 \int^{\infty}_0 & \, \frac {1}{t}\, \left |f\left(\frac{z}{t}\right)\right|\, d\mu(t) \leq
  \int^{\infty}_0 \, \frac {1}{t} \,e^{\frac{\alpha}{2} | \frac{z}{t}|^2} \, d\mu(t) \, \|f\|_{F^{\infty}_{\alpha}}\\
  &\\
  & = \int_0^{\infty} \,  \frac {1}{t} \, \left( \sum_{n\geq 0}\, \frac{1}{n!}\, \frac{{\alpha}^n\,|z|^{2n}}{2^n\, t^{2n}}\right)\, d\mu(t) \, \|f\|_{F^{\infty}_{\alpha}} \\
  &\\
  & =\, \left[\,\sum_{n\geq 0}\, \frac{1}{n!} \, \left( \int_0^{\infty} \,  \frac {1}{t^{2n+1}}\,d\mu(t) \right) \, \frac{{\alpha}^n\,|z|^{2n}}{2^n}\,\right]\, \, \|f\|_{F^{\infty}_{\alpha}}\\
  &\\
  & \leq \, \left( \sup_{n \in {\mathbb{N}}}\, \int_0^{\infty}\, \frac {1}{t^{2n+1}}\,d\mu(t) \right) \,\, e^{\frac{\alpha}{2} |z|^2}\,\, \|f\|_{F^{\infty}_{\alpha}}\,.
\end{align*}
Therefore,  under the condition \\
\begin{equation*}
\sup_{n \in {\mathbb{N}}}\, \int_0^{\infty}\, \frac {1}{t^{2n+1}}\,d\mu(t)\,< \infty
\end{equation*}\\
the Hausdorff operator $\mathcal{H}_{\mu}$ is well defined on any Fock space $F^p_{\alpha}\,,\,\,p\in [1,\infty]$\,.
It's not difficult to check that the latter condition is equivalent to \\
$$
\sup_{n \in \mathbb {N}}\, \int_0^{\infty}\, \frac {1}{t^{n+1}}\,d\mu(t)\,< \infty\,\,\,.
$$\\
Furthermore, on the one hand we get, that
$$
\sup_{n \in \mathbb {N}}\, \int_1^{\infty}\, \frac {1}{t^{n+1}}\,d\mu(t)\,< \infty
$$
and on the other hand, applying a standard measure theoretic argument to the  condition
$$
\sup_{n \in \mathbb {N}}\, \int_0^{1}\, \frac {1}{t^{n+1}}\,d\mu(t)\,< \infty\,,
$$
results that
$$
\mu(0,1)=0 \,.
$$
Therefore,  it only make sense to  consider  positive Borel  measures $\mu$ with that property
and
$$
\sup_{n \in {\mathbb{N}}}\,\int_1^{\infty} \, \frac {1}{t^{n+1}}\,d\mu(t)\,< \infty\,.
$$
If $f\in F^p_\alpha$, then the operator is written as
$$
\mathcal{H}_{\mu}(f)(z)=\int_1^{\infty}\,\frac 1t \,f \left(\frac zt\right)\, d\mu(t) \,,\quad z\in \mathbb D\,.
$$

For the proof of  Theorem \ref{boundedness} we need the following complex interpolation argument for the  Fock spaces.
 Suppose  that $1 \leq p_0 < p_1 \leq \infty$ and $0 \leq \theta \leq 1$. Then
$$
[F^{p_{0}}_{\alpha,} ,F^{p_{1}}_{\alpha,}]_{\theta} = F^{p}_{\alpha,},
$$
where
$$\frac{1}{p}=\frac{1-\theta}{p_{0}}+\frac{\theta}{p_{1}}\,\,\,.$$
In terms of operator theory, it can be interpreted as  a boundedness property of  an operator acting on the spaces under discussion. We state it for the operator of our interest. Assume that
$$ \|\mathcal{H}_{\mu}\|_{F^{p_{0}}_{\alpha} \to F^{p_{0}}_{\alpha} } \,\leq \,  M_0 $$
and that
$$ \|\mathcal{H}_{\mu}\|_{F^{p_{1}}_{\alpha} \to F^{p_{1}}_{\alpha} } \, \leq \, M_1\,\,. $$\\
Then $ \mathcal{H}_{\mu}$ is bounded in
$$[F^{p_{0}}_{\alpha} ,F^{p_{1}}_{\alpha}]_{\theta}\,=F^{p}_{\alpha}\,\,, \quad \frac{1}{p}=\frac{1-\theta}{p_{0}}+\frac{\theta}{p_{1}} $$
with norm\\
$$ \|\mathcal{H}_{\mu}\|_{[F^{p_{0}}_{\alpha,} ,F^{p_{1}}_{\alpha,}]_{\theta}} \leq M_0^{1-\theta}M_1^{\theta}\,. $$\\
 For a complete presentation of the complex interpolation theory for the Fock spaces  see  \cite{Zh2}.

 \begin{Th}\label{boundedness}
	Let $1\leq p\leq  \infty$, $\alpha>0$ and $\mu$ be a positive Borel measure on $(0,\infty)$ such that
	(\ref{well definition1}),(\ref{well definition2}) hold. Then the  Hausdorff operator $\mathcal{H}_{\mu}$ is bounded on $F^{p}_{\alpha}$\,.
	Moreover \\
	\begin{equation*}
	||\mathcal{H}_{\mu}||_{F^{p}_{\alpha}\to F^{p}_{\alpha}}=\sup\left\{\int_{1}^{\infty}\frac{1}{t^{n+1}}\,d\mu(t)\,: n \in\mathbb{N}\right\}
	\,,\,\,\,\,\,1\leq p\leq \infty\,\,.
	\end{equation*}
\end{Th}
 \begin{proof}
 Assume that the measure $\mu$ satisfies  conditions (\ref{well definition1}),  (\ref{well definition2}).
 First we deal with case $p=1$.  Let an $f \in F^1_{\alpha}$ then
 \begin{align*}
\|\mathcal{H}_{\mu}(f)\|_{F^1_{\alpha}}&=\, \frac{a}{2\pi}\, \int_{\mathbb C} \, |\mathcal{H}_{\mu}(f)(z)|\, e^{-\frac{\alpha}{2}|z|^2}\, dA(z)\\
&\\
& =  \, \frac{a}{2\pi}\, \int_{\mathbb C} \, \left|\int_{1}^{\infty}\frac{1}{t}f\left(\frac{z}{t}\right)\,d\mu(t)\right|\, e^{-\frac{\alpha}{2}|z|^2}\, dA(z)\\
&\\
& \leq  \, \frac{a}{2\pi}\, \int_{1}^{\infty}\frac{1}{t}\,\,\int_{\mathbb C} \left | f\left(\frac{z}{t}\right)\right| \,  e^{-\frac{\alpha}{2}|z|^2}\, dA(z)\,\,d\mu(t)\\
&\\
& = \, \frac{a}{2\pi}\, \int_{1}^{\infty}\frac{1}{t}\,\,\int_0^{\infty}\,\int_{0}^{2\pi}\left | f\left(\frac{r}{t}\,e^{i\theta}\right)\right| \, d\theta e^{-\frac{\alpha}{2}r^2}\, r\, dr\,\,d\mu(t)\\
&\\
& \leq \frac{a}{2\pi}\, \int_{1}^{\infty}\frac{1}{t}\,\,\int_0^{\infty}\,\int_{0}^{2\pi}\left | f\left(r\,e^{i\theta}\right)\right| \, d\theta e^{-\frac{\alpha}{2}r^2}\,r\, dr\,\,d\mu(t)\,,
\end{align*}
since the integrals
$$
M(f,s)=\int_{0}^{2\pi} \left | f\left(s\,e^{i\theta}\right) \right |\,d\theta
$$
are increasing functions of s.\\
So
\begin{align}\label{upper norm F1}
 \|\mathcal{H}_{\mu}(f)\|_{F^1_{\alpha}}&\leq \,\, \sup_{n\in \mathbb N} \left\{\,\int_1^{\infty}\,\,\frac{1}{t^{n+1}}\,\,d\mu(t) \right\}\,\,
 \| f \|_{F^1_{\alpha}}\,.
 \end{align}

 On the other hand, the argument for the well definition of the operator provides the following estimation

 \begin{align}\label{upper norm Finfty}
||\mathcal{H}_{\mu}(f)||_{F^{\infty}_{\alpha}} \,\leq
\,\sup_{n\in \mathbb N} \left\{ \int_{1}^{\infty}\frac{1}{t^{n+1}}\, d\mu(t)\right\}\,\,\|f\|_{F^{\infty}_{\alpha}}\,.
 \end{align}
  \\
  Combining (\ref{upper norm F1}) and \,(\ref{upper norm Finfty}) with the interpolation property, we get that
  \\
   \begin{align}\label{upper bound}
   ||\mathcal{H}_{\mu}||_{{F^{p}_{\alpha}}\to F^{p}_{\alpha}} \,\leq
\,\sup_{n\in \mathbb N} \left\{ \int_{1}^{\infty}\frac{1}{t^{n+1}}\, d\mu(t)\right\}\,,\quad 1\leq p\leq \infty\,.
     \end{align}
      \\
   \par It turns out that
   \\
   \begin{align*}
   ||\mathcal{H}_{\mu}||_{{F^{p}_{\alpha}}\to F^{p}_{\alpha}}\,=\,\sup_{n\in \mathbb N} \left\{ \int_{1}^{\infty}\frac{1}{t^{n+1}}\, d\mu(t)\right\}\,,\quad 1\leq p\leq \infty\,.
    \end{align*}
    \\
   This is a consequence of  the following observation. The monomials
   \\
   \begin{equation*}
   u_n(z) \,=\, z^n\,,\quad n\in \mathbb N\,
   \end{equation*}
   \\
    serve as test functions for  $F^p_{\alpha}$, so for any $F^p_{\alpha}$, \,$p\in[1,\,\infty]$,\,and
   \\
    \begin{equation*}
   {\mathcal H}_{\mu}(u_n )(z) \,\,=\,\,\int_{1}^{\infty} \,\frac{1}{t^{n+1}}\,d\mu(t)\,\,u_n(z)\,,\quad n\in \mathbb N\,.
   \end{equation*}
   \\
   Therefore
   \begin{align*}
\|\mathcal{H}_{\mu} \|_{F^{p}_{\alpha}\to F^{p}_{\alpha}}
&\geq \,\int_{1}^{\infty}\frac{1}{t^{n+1}}\, d\mu(t)\,,
\end{align*}
for every $n \in \mathbb N$.

 \end{proof}

\section{Compactness }
In this section we deal with the problem of compactness. This will be a two step procedure. First we consider the Hilbert space case and we use  well known arguments of operator theory on Hilbert spaces. These are stated below
for the space $F^{2}_{\alpha}$.

The necessity is proved using  the  property, that  compactness can be equivalently determined by the condition
\begin{equation}\label{weak compactness}
\|T(f_n)\|_{F^{2}_{\alpha}} \to 0\,, \quad \text{whenever}\,\, f_n \,\,\to 0 \,\,\text{weakly in}\,F^{2}_{\alpha}\,.
\end{equation}
On the other hand  the sufficient condition results from a well known property according to which
if $\{T_n\}_n$ is a sequence of compact operators on $F^{2}_{\alpha}$ such that
\begin{equation}\label{finite rank}
\|T-T_n\| \,\to 0 \,, \qquad n\to \infty,
\end{equation}
then $T$ is compact on $F^{2}_{\alpha}$.

\begin{Th}\label{comp F(2,a)}
	Assume that $\mu$  is a positive Borel measure on $(0,\infty)$ such that
	(\ref{well definition1}), (\ref{well definition2}) are true. The  Hausdorff operator $\mathcal{H}_{\mu}$ is compact   on $F^{2}_{\alpha},\,\, \alpha>0,\,\,$  if and only if
	\\
	\begin{equation*}
	\lim_{n\to \infty}\int_{1}^{\infty}\frac{1}{t^{n+1}}\,d\mu(t)=0.
	\end{equation*}
\end{Th}

\begin{proof}
	
	Suppose that $\mathcal{H}_{\mu}$ is compact on  $F^2_{\alpha}$.
	We consider the orthonormal basis of $F^2_{\alpha}$
	\begin{equation}\label{basis}
	e_{n}(z)=\sqrt{\frac{\alpha^{n}}{n!}}\,z^{n}\,, \quad n \in \mathbb N\,.
	\end{equation}
	It is true that
	$$ e_{n}\to 0  \quad \text{weakly in} \,\,\,F^{2}_{\alpha}\,.$$
	This combined with the assumption, that  $\mathcal{H}_{\mu}$ is compact
	and the property (\ref{weak compactness}) imply that
	\\

$$
||\mathcal{H}_{\mu}({e}_{n})||_{F^{2}_{\alpha}}\to 0\,, \quad \text{as} \quad n\to \infty\,.
$$
Since
\begin{equation}\label{expansion1}
\mathcal{H}_{\mu}({e}_{n})(z)=\int_{1}^{\infty}\frac{1}{t^{n+1}}\mu(t)\,{e}_{n}(z)\,,
\end{equation}
it is immediate that
$$
\lim_{n\to \infty}\int_{1}^{\infty}\frac{1}{t^{n+1}}\,d\mu(t)=0.
$$

Conversely suppose that
$$
\lim_{n\to \infty}\int_{1}^{\infty}\frac{1}{t^{n+1}}\,d\mu(t)=0\,.
$$
\\
 For a given $\varepsilon>0$,  there is a $n_0 \in \mathbb N $  such that, for every $n>n_0$, we have
 \\
$$\int_{1}^{\infty}\frac{1}{t^{n+1}}\,d\mu(t)<\varepsilon\,.$$
\\
Consider now  the finite rank operator
\\
 $$
\mathcal{H}_{\mu,k }(f)(z)=\sum_{n=0}^{k}\,\int_{1}^{\infty}\frac{1}{t^{n+1}}\,d\mu(t)\,a_{n}\, z^{n}\,,\quad k\in \mathbb N
$$
\\
where  $f(z)=\sum a_n z^n\in F^{2}_{\alpha}$\,. If we take into account (\ref{norm Hilbert}),  we get that

\begin{align*}
||\mathcal{H}_{\mu}(f)&-\mathcal{H}_{\mu,k}(f)||_{F^{2}_{\alpha}}^{2}
=\sum_{n=k +1}^{\infty}\left(\int_{1}^{\infty}\frac{1}{t^{n+1}}\,d\mu(t)\right)^{2}|a_{n}|^{2}\,\frac{\alpha^{n}}{n!}\\
&\\
& <  \varepsilon\,\,\, \sum_{n=k+1}^{\infty} |a_{n}|^{2} \,\frac{\alpha^{n}}{n!}
< \varepsilon\,\,\, ||f||_{F^{2}_{\alpha}}^{2}\,, \quad \text{for every}\,\,\, k>n_0\,.
\end{align*}

As a consequence of the condition (\ref{finite rank}) we prove,  that  $\mathcal{H}_{\mu}$ is compact.\\
\end{proof}

Now, in order to complete the scene for the compactness,  we  have to pass from the case $p=2$ to any $p\in(1,\infty)$. As in the case of boundedness we will apply an interpolation argument for compact operators. The study of interpolation properties of compact operators is a classical area   with  applications  to other branches of analysis. The behavior of compact operators under the complex interpolation method is one of  the main open problems on the subject.
In oder to confront the problem of compactness for the spaces $F^{p}_{\alpha},\,p \neq 2,$ we will employ Theorem $9$ in \cite{CwKal} which is an interpolation technique for Banach spaces, that have the UMD
property that is the property of Unconditional Martingale Differences. For more information on the subject we suggest to the interested reader the \cite{HNVW}.
Since Hilbert spaces have this property,
we state Theorem $9$ \cite{CwKal} adapted to our case.

\begin{Theorem*}
	Let $(X_0, X_1)$ be a Banach couple and $X_0$ be a Hilbert  space. If,  in the following cases ,
	\begin{equation*}
	T : X_0+X_1 \to X_0 +X_1 \,,
	\end{equation*}
	\begin{equation*}
	T : X_j \to X_j \,, \quad\quad j=0,1
	\end{equation*}
	$T$ acts as a bounded  operator
	and
	\begin{equation*}
	T : X_0\to X_0
	\end{equation*}
	  is a compact operator then
	\begin{equation*}
	T :[X_0, X_1]_{\theta}  \to [X_0, X_1]_{\theta}
	\end{equation*}
	is compact  for any $\theta \in (0,1)$ .
\end{Theorem*}
Applying this, we complete the characterization of compactness for  the Hausdorff operator on any $F^{p}_{\alpha}\,,\,p\in(1,\infty)$.
\begin{Th}\label{comp F(p,a)}
	Let $p\in(1,\infty)$ and $\alpha>0$. Assume that $\mu$  is a positive Borel measure on $(0,\infty)$ such that
	(\ref{well definition1}),(\ref{well definition2}) are true. Then the  Hausdorff operator $\mathcal{H}_{\mu}$ is compact   on $F^{p}_{\alpha}$  if and only if
	\\
	\begin{equation*}
	\lim_{n\to \infty}\int_{1}^{\infty}\frac{1}{t^{n+1}}\,d\mu(t)=0.
	\end{equation*}
\end{Th}
\begin{proof}
The sufficiency is an application of the  interpolation argument, we have presented above.
 Choosing as $X_0 = F^{2}_{\alpha}$ and as $X_1=  F^{p_0}_{\alpha}$,  for  $p_0=1$ or $p_0= \infty$,
 recalling that
$$
[F^{2}_{\alpha,} ,F^{p_0}_{\alpha,}]_{\theta} = F^{p}_{\alpha,},
$$
where
$$\frac{1}{p}=\frac{1-\theta}{2}+\frac{\theta}{p_{0}}\,\,\,,$$
and assuming the condition
\begin{equation*}
	\lim_{n\to \infty}\int_{1}^{\infty}\frac{1}{t^{n+1}}\,d\mu(t)=0
	\end{equation*}
we ensure all the prerequisites of the latter theorem so to conclude that
\begin{equation*}
	T :F^{p}_{\alpha}  \to F^{p}_{\alpha} \quad p\in (1,\infty)\,.
	\end{equation*}
is compact.

For the converse, suppose that $\mathcal{H}_{\mu}$ is compact on  $F^p_{\alpha}$.
We consider the normalized functions
\begin{equation}\label{monomials normalazed}
  \widetilde{e}_{n}(z)=\frac{(\alpha p)^{n/2}}{\Gamma(\frac{np}{2}+1)^{1/p}\, 2^{n/2}}\,z^{n}\,, \quad  n\in \mathbb N\,.
  \end{equation}
Then it is easy to see (\cite{CoMac}[Corollary 1.3]) that
$$ \widetilde{e}_{n}\to 0  \quad \text{weakly in} \,\,\,F^{p}_{\alpha}\,$$
and thus as in Theorem \ref{comp F(2,a)} turns out that
$$
\lim_{n\to \infty}\int_{1}^{\infty}\frac{1}{t^{n+1}}\,d\mu(t)=0\,.
$$
\end{proof}

\section{Schatten Classes}

The objective of this last part  is the study of the membership   of $\mathcal{H}_{\mu}$ in the Schatten Classes  $\mathcal{S}_{p}(F^2_{\alpha})$ .
 We begin with an  $\mathcal{H}_{\mu}$   compact  on $F^2_{\alpha}$, that is
  $$
\lim_{n\to \infty}\int_{1}^{\infty}\frac{1}{t^{n+1}}\,d\mu(t)=0\,.
$$
\\
  From one point of view,   (\ref{expansion1}) implies that
  \begin{equation}\label{point spectrum 1}
  \left\{\,\,\int_{1}^{\infty}\frac{1}{t^{n+1}}\,d\mu(t)\,,\,\, n=0,1,2,...\,\,\right\}\subset\sigma_{p}(\mathcal{H}_{\mu})\,,
  \end{equation}
 where  $\sigma_{p}(\mathcal{H}_{\mu})$  is the point spectrum of the Hausdorff operator.
 On the other hand,  it  justifies that  $\mathcal{H}_{\mu}$ is a diagonal operator with respect to the orthonormal basis
 $$e_{n}(z)=\sqrt{\frac{\alpha^{n}}{n!}}\,z^{n}\,, \quad n \in \mathbb N  $$

 of  $F^2_{\alpha}$.
 The terms of the  sequence
  $$
  \left\{\,\,\int_{1}^{\infty}\frac{1}{t^{n+1}}\,d\mu(t)\,,\,\, n=0,1,2,...\,\,\right\}
  $$
  are the diagonal entries of it's matrix. Being $\mathcal{H}_{\mu}$  compact and diagonal operator,   implies

\begin{align}
  \sigma_{p}(\mathcal{H}_{\mu})&\subset \overline{\left\{\,\,\int_{1}^{\infty}\frac{1}{t^{n+1}}\,d\mu(t)\,,\,\, n=0,1,2,...\,\,\right\}}\nonumber\\
  &\nonumber \\
 &=\left\{\,\,\int_{1}^{\infty}\frac{1}{t^{n+1}}\,d\mu(t),\,\, n=0,1,2,...\,\,\right\}\cup\left\{0\right\}. \label{point spectrum 2}
  \end{align}
  As a reference for the spectral theory of diagonal operators we propose \cite{Hal}\,.

 Combining  (\ref{point spectrum 1}) and (\ref{point spectrum 2})  with the fact that a diagonal  is a normal operator give us the right to employ one last argument according to which the  membership in the  Schatten Class $\mathcal{S}_{p}$  of a compact and normal operator  is equivalent to the p-summability of its eigenvalues. For example see \cite{GoKr},\cite{Zh1}\,.

  Now we are in position to state the main result of this section.
\begin{Th}\nonumber
Let $\mu$ be a postive Borel measure on $(0,\infty)$ such that
 (\ref{well definition1}),(\ref{well definition2}) hold.
 The Hausdorff operator $\mathcal{H}_{\mu}$  belongs to the Schatten class $\mathcal{S}_{p}(F^2_{\alpha})$,\,  $p\in(0,\infty),$  if and only if
$$
\sum_{n=0}^{\infty}\left(\int_{1}^{\infty}\frac{1}{t^{n+1}}\,d\mu(t)\right)^{p}<\infty.
$$
\end{Th}

\Addresses

\end{document}